\documentclass[11pt,twoside]{amsart}
\usepackage{amsmath}
\usepackage{amsthm}
\usepackage{amsfonts}
\usepackage{amssymb}
\usepackage{dutchcal}
\usepackage{color}
\usepackage{mathrsfs}
\usepackage{mathtools}
\usepackage{upgreek}
\DeclareMathAlphabet{\mathcal}{OMS}{cmsy}{m}{n}

\usepackage{enumitem}
\usepackage{indentfirst}
\usepackage[normalem]{ulem}
\usepackage{tikz}
\usepackage{tikz-cd}	 
\usepackage{array}  
\newcolumntype{L}{>{$}l<{$}} 
\usepackage{extpfeil}
\usepackage{verbatim}
\usepackage{hyperref}
\usepackage{cleveref}
\usepackage{lipsum}
\usepackage{comment}
\usepackage{filecontents}
\usepackage[toc]{appendix} 
\usepackage[style = alphabetic, doi=false,isbn=false,url=false]{biblatex}

\newcommand\blfootnote[1]{
	\begingroup
	\renewcommand\thefootnote{}\footnote{#1}
	\addtocounter{footnote}{-1}
	\endgroup
}

\def \mbc {\mathbb{C}}
\def \mbp {\mathbb{P}}
\def \mso {\mathcal{O}}
\def \mcf {\mathcal{F}}
\def \mcd {\mathcal{D}}

\def \mce {\mathcal{E}}
\def \mbp {\mathbb{P}}
\def \mbf {\mathbb{F}}
\def \mbz {\mathbb{Z}}
\def \pic {\textup{Pic}}

\def \ch {\textup{CH}}
\def \ra {\rightarrow}

\renewbibmacro{in:}{} 
\DeclareFieldFormat{pages}{#1}

\newtheorem{thm}{Theorem}[section]
\newtheorem{lem}[thm]{Lemma}
\newtheorem{prop}[thm]{Proposition}
\newtheorem{cor}[thm]{Corollary}
\newtheorem{conj}[thm]{Conjecture}
\newtheorem*{conj-no}{Conjecture}

\theoremstyle{definition}
\newtheorem{defn}[thm]{Definition}
\newtheorem{rmk}[thm]{Remark}
\newtheorem{exmp}[thm]{Example}

\numberwithin{thm}{section}

\title{The Curves of Elliptic Nodes on K3 Surfaces}
\author{Huang Jiexiang}

\hypersetup{
	colorlinks = true,
	linkcolor=black,
	citecolor=black
}

\begin{document}
	\begin{abstract}
		Let $(X,L)$ be a polarized K3 surface of genus $g$ and $C_{en} \subset X$ be the  curve of singular points of nodal elliptic curves in $|L|$. When $(X,L)$ is generic of genus two, Huybrechts observed that the curve $C_{en}$ is a constant cycle curve and conjectured that this remains true for higher genus cases. In this note, we show that the conjecture holds true for polarized K3 surfaces $(X,L)$ lying in a locus of codimension one in the moduli space of polarized K3 surfaces of genus $g$ for every $g > 2$. 
	\end{abstract}
	
	\maketitle

	\section{Introduction}
	Constant cycle curves on complex projective K3 surfaces are curves whose points represent the same class in $\ch_0(X)$. Rational curves on K3 surfaces are the  simplest examples and those of higher genus were sysmetically studied in \cite{Huybrechts2014}. Among them a  particularly interesting example arises from  curves in the fixed locus of a non-symplectic automorphism of a K3 surface (see \cite[Prop.~7.1]{Huybrechts2014}). 
	
	In this note, we study further concrete examples of constant cycle curves on K3 surfaces, especially  those generic  in the moduli space of polarized K3 surfaces. A natural candidate for a generic polarized K3 surface $(X,L)$ of genus $g$, suggested by Huybrechts, is constructed as follows. Recall from \cite{MoriMukai} that for  a generic K3 surface there exists a one-dimensional family of nodal elliptic curves in $|L|$. 
	 Their singularities sweep out a curve $C_{en}$ (up to taking the closure)  on $X$ and we call it the  \textit{curve of elliptic nodes} of $(X,L)$. 
	 
	 When $g =2$, Huybrechts observed that the curve $C_{en}$  coincides with the fixed locus of the natural involution on $X$ (see \Cref{geointer_dbplane}), thus is a constant cycle curve. However, a similar geometric interpretation for $C_{en}$ no longer exists when the genus of $(X,L)$  becomes higher. We wonder whether the  curve of elliptic nodes $C_{en}$  in the cases $g \geqslant 3$ nevertherless still offers new examples of constant cycle curves. This leads to the following conjecture.
	 
\begin{conj}[Huybrechts] \label{intro_conj}
		Let $(X,L)$ be a  polarized K3 surface of genus $g \geqslant 3 $. Then the  curve $C_{en}$ of elliptic nodes  of $(X,L)$, if nonempty, is a constant cycle curve.
	\end{conj}
	Note that here we allow non-generic polarized K3 surfaces. For details we refer to \Cref{formaldef} and \Cref{nonempty}.
	
	\blfootnote{The author is supported by the ERC Synergy Grant HyperK (ID 854361).}

	In this paper, we give the following partial answer to \Cref{intro_conj} (see  \Cref{proof} for the proof).
	
	\begin{thm}\label{mainthm}
		Let $\mcf_g$ be the moduli space of polarized K3 surfaces of genus $g$. For each $g>2$, there exists a locus of codimension one in $\mcf_g$ for which \Cref{intro_conj} holds.
	\end{thm}

	This result  is obtained by studying the curves of elliptic nodes associated with hyperelliptic K3 surfaces.  According to \cite{Reid1976HyperellipticLS}, these special K3 surfaces can always be realized as double covers of Hirzebruch surfaces or $\mbp^2$ and  form an $18$-dimensional locus in each $\mcf_g$ when $g \geqslant 3$, see \Cref{hek3_review} for a brief summary of the classification theory. Similar to the   case $g=2$, the  curve $C_{en}$ associated with a hyperelliptic K3 surface can  be identified with the fixed locus of the hyperelliptic involution (see \Cref{proof}).

	Note that a priori the  curve of elliptic nodes $C_{en}$ can be empty for non-generic polarized K3 surfaces. However, it turns out  that at least for a generic hyperelliptic K3 surface $(X,L)$ of the type that we are concerned about, the construction of $C_{en}$ is valid (see \Cref{nonempty}).

	For generic polarized K3 surfaces of genus $g>2$, \Cref{intro_conj} is still open. The major difficulty is that we lack an alternative geometric interpretation for the distinguished curve $C_{en}$ that would enable us to compare the cycle classes of closed points on $C_{en}$ in $\ch_{0}(X)$, even if the projective model of $X$ is known, e.g.\ when $(X,L)$ is generic of small genus. 
	
	Finally, we point out that based on our knowledge of $\ch_0(X)$,  hyperelliptic K3 surfaces are almost all the examples (in a naive sense) for which we can prove \Cref{intro_conj}. This aspect is  discussed in \Cref{discussion}.
	
	\hfill\\
	\textbf{Acknowledgement.}
	This paper is part of my master thesis at Universität Bonn. I would like to thank my advisor Daniel Huybrechts for suggesting me this interesting topic, inspiring discussions and  helpful comments on improving this paper.
	\hfill\\
	\hfill\\
	\textbf{Notations and Conventions.} We always work over the base field $\mbc$. A nodal ellipic curve in our convention is an integral nodal projective curve of geometric genus one.
	The coarse moduli space of polarized K3 surfaces of genus $g$ is denoted by $\mcf_g$. We say a property holds for a \textit{generic}  polarized K3 surface of genus $g$ if it holds for points $(X,L)$ in some non-empty Zariski open subset of $\mcf_g$. 
		
	\section{Preparations} \label{sec_prep}
	The aim of this section is to motivate the study of \Cref{intro_conj}. To be specific, we give the formal definition of the curve $C_{en}$ of elliptic nodes of a generic polarized K3 surface $(X,L)$ of genus $g \geqslant 2$ and then present Huybrechts' proof of the conjecture for  the simplest case $g = 2$.

	\subsection{Constant cycle curves}
	Let us first review some basic facts about constant cycle curves  following \cite{Huybrechts2014}. 	Throughout this part, $X$ is always a projective K3 surface over $\mbc$.

	\begin{defn}
		A curve $C \subset X$ is a \textit{constant cycle curve} if $[x] = [y]$ in $\ch_0(X)$ for any two  closed points $x, y \in C$.
	\end{defn}

	Recall from \cite{beauville2004chow} that the group $\ch_0(X)$ contains a canonical class $c_X \in \ch_0(X)$, called the \textit{Beauville--Voisin class}, which can be realized by any point on a rational curve $C \subset X$ and is  distinguished in the sense that the image of the intersection product
	$$
	\pic(X) \otimes \pic(X) \ra \ch_0(X)
	$$
	is contained in $\mbz \cdot c_X$.  Then indeed the  cycle class represented by a point of a  constant cycle curve $C \subset X$ is exactly the Beauville--Voisin class $c_X$ (see  \cite[Lem.~2.2]{voisin2015rational}). In addition,  by \cite[Lem.~10.7]{voisin2003bookvolII},  to show that a curve is a constant cycle curve, it suffices to verify that $[x] = [y] \in \ch_0(X)$  for any two points $x,y$ in an open dense subset $U \subset C$.

	The following result  \cite[Prop.~7.1]{Huybrechts2014} provides a method to produce constant cycle curves on K3 surfaces that admit non-symplectic automorphisms.
	
	\begin{prop}
	\label{intro_prop_nonsymplectic}
		Let $f$ be a non-symplectic automorphism of finite order $n$ of a K3 surface $X$. Then any curve $C \subset X$ contained in the fixed locus of $f$ is a constant cycle curve.
	\end{prop}
	
	Observe that by  \cite{MumfordChowzero} and the verified case of Bloch's conjecture in \cite{Bloch1976ZeroCO}, an automorphism $f$ is non-symplectic if and only if the quotient surface $\bar{X} \coloneqq X/\langle f \rangle$ satisfies $\ch_0(\bar{X}) \simeq \mbz$. This  induces the following toy example for the above construction.
	\begin{exmp}[{\cite[Sec.~7.1]{Huybrechts2014}}] \label{exmp_intro}
		Let $X$ be a generic polarized K3 surface of genus two. Then $X$ admits a double cover $\pi \colon X \ra \mbp^2$, ramified over a smooth sextic $B \in |\mso_{\mbp^2}(6)|$. The naturally induced involution $i$ on $X$ fixes the curve $D \coloneqq \pi^{-1}(B)$. Since $\ch_0(\mbp^2) \simeq \mbz$, the remark above implies that $i$ is non-symplectic, and hence $D$ is a constant cycle curve.
	\end{exmp}
	
	\subsection{An alternative description of the fixed curve $D$}
	Note that when $g \geqslant 3$, \Cref{intro_prop_nonsymplectic} merely produces constant cycle curves on special K3 surfaces, i.e.\ those non-generic in the moduli space. However, Huybrechts observed that the fixed curve $D$ has another geometric description that can be generalized to polarized K3 surfaces of higher genus. We present his idea here in details and expect to obtain from it new examples of constant cycle curves on generic polarized K3 surfaces of genus $g \geqslant 3$.

	\begin{lem}[Huybrechts] \label{geointer_dbplane}
		Let $\pi \colon X \ra \mbp^2$ be a double cover  ramified over a smooth sextic $B \in |\mso_{\mbp^2}(6)|$ and let $L = \pi^{\ast}\mso_{\mbp^2}(1)$. Then the constant cycle curve $D \coloneqq \pi^{-1}(B)$ (see \Cref{exmp_intro}) can  be alternatively described as the closure of the following set
		\begin{center}
			\textup{	\{$P \in X \mid P $ is the nodal point of some nodal elliptic curve $E \in |L|$\}.}
		\end{center}
	 Note that in our setting, any nodal elliptic curve in $|L|$ has exactly one node.
	\end{lem}

	\begin{proof}[Proof]
	We briefly explain the picture behind the idea. Observe that there exists a bijection between the following two sets 
		\begin{enumerate}[label=(\roman*)]
			\item lines in $|\mso_{\mbp^2}(1)|$ that are simply tangent to $B$ at one point and intersect $B$ transversally at the other intersection points
			\item  nodal elliptic curves in $|L| $
		\end{enumerate} 
		via the pullback $\pi^\ast \colon |\mso_{\mbp^2}(1)| \ra |L|$ of linear systems. Moreover, the tangent point of  a line $l \in |\mso_{\mbp^2}(1)|$ in (i) with the branch locus $B$ is pulled back by $\pi$ to the node of the nodal elliptic curve $\pi^{\ast}l \in |L|$. Since a generic tangent line of $B$ satisfies the intersection condition in (i), the singularities of nodal elliptic curves in $|L|$ go through all but at most finitely many points on the branch locus $B$. We conclude  by taking the closure of this set of nodal points in $X$.
	\end{proof}
	\subsection{The curve of elliptic nodes of $(X,L)$} \label{formaldef}

	 Based on the previous geometrical observation in \Cref{geointer_dbplane} for the fixed curve $D$ (see \Cref{exmp_intro} for the construction), we can now easily generalize the construction of the curve $D$ to any polarized K3 surface $(X,L)$ of genus $g \geqslant 3$.  
	
	Let us first fix some necessary notations.  
	We say a nodal curve $C \subset X$ is \textit{$\delta$-nodal}, if it has exactly $\delta$ nodal singularities and is smooth elsewhere. The locus of integral $\delta$-nodal curves in $|L|$, denoted by $V_{L,\delta}$, is called the \textit{Severi variety} (of $\delta$-nodal curves) of the pair $(X,L)$. Recall from \cite[Ex.~1.3]{SernesiChiantini1996NodalCO} that for any $1 \leqslant \delta \leqslant g$, the Severi variety $V_{L,\delta} \subset |L|$ is a smooth subvariety of  dimension $g-\delta$ once it is nonempty. In particular, nodal elliptic curves in $|L|$ are $(g-1)$-nodal and form a one-dimensional family, once there exists any of them in $|L|$. 
	
	\begin{rmk}
		The nonemptieness of $V_{L,\delta}$ can be ensured by e.g.\ assuming that $(X,L)$ is generic in the moduli space $\mathcal{F}_g$ of polarized K3 surfaces of genus $g$ (see \cite{MoriMukai}). However, for special polarized K3 surfaces in the moduli space $\mcf_{g}$, for example,  certain lattice polarized K3 surface $(X,L)$, the linear system $|L|$ may also contain  an integral nodal rational curve $C$  (see \cite{chengounelasliedtke2022curves}). Consequently, by the deformation theory of nodal curves on K3 surfaces, one can always smooth one node of $C$ and  obtain a one-dimensional family of nodal elliptic curves in $|L|$ as well. We will come back to this point in \Cref{nonempty}.
	\end{rmk}

	\begin{defn} \label{def_Cd}
		Let $(X,L)$ be a polarized K3 surface of genus $g$.  \textit{The curve of elliptic nodes}  $C_{en}$ of $(X,L)$ is defined as the one-dimensional component of the closure of the following set
		\begin{center}
			\{$P \in X \mid P $ is the nodal point of some nodal elliptic curve $E \in |L|$\}
		\end{center}
		with the reduced scheme structure (we allow $C_{en}$ to be empty if there exists no nodal elliptic curves in $|L|$) .
	\end{defn}
	
	\begin{rmk}
		The closure of the above set of singularities has dimension one once the Severi variety $V_{L,g-1}$ is nonempty, but we do not know if the closure is always irreducible or even at least of pure dimension one, since $V_{L,g-1} \subset |L|$ may have base point at one of the nodal points of some nodal elliptic curve in $|L|$. Therefore, we discard  discrete points that possibily exist manually in the definition of $C_{en}$. This is only known to be unnecessary when $(X,L)$ is generic of genus three by \cite{witaszek2014geometry}:  There the curve of elliptic nodes $C_{en}$  of a smooth quartic is called the ``double-cover curve" and its geometry was studied in details. For instance, $C_{en}$ is an irreducible curve of degree $320$, genus $1281$ with only nodal and cusp singularities.
	\end{rmk}
	
	Based on \Cref{exmp_intro} and  \Cref{geointer_dbplane}, we obtain
	\begin{cor}
	The curve of elliptic nodes $C_{en}$ of a generic polarized K3 surface $(X,L)$ of genus two is a constant cycle curve.
	\end{cor}
	
	We end this section with some useful observations on sufficient conditions for \Cref{intro_conj}, i.e.\ for $C_{en}$ to be a constant cycle curve.
	
	Let $(X,L)$ be a polarized K3 surface of genus $g > 2$ with nonempty Severi variety $V_{L,g-1}$. For a nodal elliptic curve $E \in |L|$, we denote its $(g-1)$ nodal points by $P_1, \ldots, P_{g-1}$ and write $\{P_{i}^{\prime},P_{i}^{\prime \prime}\}$ for the preimage of $P_i$ under the normalization map $\nu \colon \tilde{E} \ra E$. Consider the following two conditions related to the curve $E$: 
	
	\begin{enumerate}[label =(\alph*)]
		\item $[P_i] = [P_j]$ in $\ch_0(X)$ for any $i \neq j$. \label{X_eq}
		\item $[P_{i}^{\prime}]+[P_{i}^{\prime \prime}] = [P_{j}^{\prime}]+[P_{j}^{\prime \prime}]$ in $\ch_0(\tilde{E})$ for any $i \neq j$.  \label{E_eq}
	\end{enumerate}

	\begin{lem} \label{easyobserv}
		Use the setting above. Then
		\begin{enumerate}[label = \normalfont(\roman*)]
			\item For any nodal elliptic curve $E \in |L|$, the condition \textup{\ref{E_eq}} implies \textup{\ref{X_eq}}.
			\item $C_{en}$ is a constant cycle curve if and only if the equation in \textup{\ref{X_eq}} holds for any nodal elliptic curve $E \in |L|$.
		\end{enumerate}
	\end{lem}
	\begin{proof}
		To see (i), we push forward the equation \ref{E_eq} along the morphism $\tilde{E} 
		\xtwoheadrightarrow{\nu} E \hookrightarrow X$ and then apply Roitman's theorem \cite{Rojtman}, i.e.\ $\ch_0(X)$ is torsion free.
		
		For the claim (ii), only  the ``if'' direction needs an explanation. First, for any nodal elliptic curve $E \in |L|$, the double point formula \cite[Thm.~9.3]{fulton1998intersection} and  Roitman's theorem  together induce the equation 
		$
		\sum_{i=1}^{g-1}[P_i] = (g-1)c_X
		$
		in $\ch_0(X)$. If \ref{X_eq} holds for $E$, then it follows immediately from the construction of $C_{en}$  that $[P] = c_X$ in $\ch_0(X)$ for any point $P$ in an open dense subset $U \subset C_{en}$ (here we use again the torsion-freeness of $\ch_0(X)$).  We conclude that  $C_{en}$ is a constant cycle curve by \cite[Lem.~10.7]{voisin2003bookvolII}.
	\end{proof}
	
	\section{The Conjecture for Hyperelliptic K3 Surfaces} \label{sec_conj_hek3}
	The aim of this section is to prove our main result \Cref{mainthm}.
	\subsection{Review on Hyperelliptic K3 Surfaces} \label{hek3_review}
	We collect here some facts in \cite{saintdonat74projmodelofK3} and \cite{Reid1976HyperellipticLS} about hyperelliptic K3 surfaces that will be used later. See also \cite[Sec.~2.3.2]{K3book_Huybrechts2016} for a summary.
	
	\begin{thm}[{\cite[Sec.~4, Cor.~5.8]{saintdonat74projmodelofK3}}] \label{saintdonat}
			Let $X$ be a K3 surface and  $L$ be a line bundle on $X$  with $(L)^2 = 2g-2 > 0$. Suppose the linear system $|L|$ has no fixed components. Then  $|L|$ has no base points and induces a morphism $\phi_L \colon X \ra \mbp^{g}$, which is 		\begin{enumerate}[label = \normalfont(\roman*)]
			\item either a generically $2:1$ morphism,
			\item or a birational morphism 
		\end{enumerate} 
	onto its image. The first case occurs if and only if any smooth irreducible member of $|L|$ is a hyperelliptic curve.
	\end{thm}
	
  \begin{defn} \label{def_HEK3}
  	Use the setting  above. In the first case, such a linear system $|L|$ is called a \textit{hyperelliptic linear system}. We say a K3 surface $X$ is \textit{hyperelliptic} if it admits a hyperelliptic linear system $|L|$, and we will call the pair $(X,L)$  a hyperelliptic K3 surface when the hyperelliptic linear system is assigned. 
  \end{defn}

	We present the  classification theory of hyperelliptic K3 surfaces by following \cite{Reid1976HyperellipticLS}. First, hyperelliptic K3 surfaces are natural generalizations of double planes, i.e.\ K3 surfaces that admit double covers of $\mbp^2$.
	\begin{thm}[\cite{Reid1976HyperellipticLS}]
		A hyperelliptic K3 surface $X$ is always a double cover of  $\mbp^2$ or the $n$-th Hirzebruch surface $\mbf_n = \mbp(\mso_{\mbp^1} \oplus \mso_{\mbp^1}(-n)) $, where $0 \leqslant n \leqslant 4$. 
	\end{thm} 
	
	Let $p \colon \mbf_n  \ra \mbp^1$ be the natural projection. Recall that $\pic(\mbf_n) \simeq \mbz e \oplus \mbz f$, where $e$ is the unique section class  representing $\mso_{p}(1)$ and $f$ is the fibre class. The two generators satisfy the  numerical relations
	\begin{equation}\label{eq_basis_relation}
		e^2 = -n, \ \ ef = 1, \ \ f^2 = 0.
	\end{equation}
	Moreover, for an integer $r$, the line bundle $\mso(e+rf)$ on $\mbf_n$ is ample if and only if $r>n$, see e.g.\ \cite[Ch.V, Cor~2.18]{hartshorne1977algebraic}. 
	Then conversely, one can construct hyperelliptic K3 surfaces as follows. 
	
	\begin{prop} \label{dbcover_is_HEK3}
		Fix an integer $n \in \{0,1,2,3,4\}$. Let $B \in |-2K_{\mbf_n}|$ be a smooth curve and   $\pi \colon X \ra \mbf_n$ be the double cover ramified over $B$. Consider the line bundle  $L_{r} \coloneqq \pi^{\ast}\mso(e+rf)$ on $X$ for an integer  $r > n$. Then $(X,L_r)$ is a polarized K3 surface of genus $g = 2r+1-n$. In addition, $(X,L_{r})$ is hyperelliptic. 
	\end{prop}
	
	The cases for $n=0$ and $1$ are especially important for us, since these two types of hyperelliptic K3 surfaces appear as a divisor in every moduli space $\mcf_{g}$ of polarized K3 surface with odd or even genus $g$ respectively when $g \geqslant 3$. 
	
	\begin{prop}[{\cite{Reid1976HyperellipticLS}}] \label{moduli_dim_heK3}
		For  $n \in \{0,1,2,3,4\}$ and $r>n$, write  $\mathcal{F}_{n,r}$  for the locus of pairs $(X,L_r)$ in the moduli space $\mcf_{g}$ of polarized K3 surfaces  of genus $g = 2r+1-n$. Then for $n = 0$ or $1$, one has
		$$
		\dim \mcf_{n,r} = \dim |-2K_{\mbf_{n}}| - \dim \textup{Aut}(\mbf_{n}) = 18.
		$$
	\end{prop} In conclusion, hyperelliptic K3 surfaces that are  double covers of $\mbf_0$ (resp.\ $\mbf_1$) form a locus of codimension one in the moduli space of polarized K3 surfaces of genus $g = 2r+1 \geqslant 3$  (resp.\ $g=2r \geqslant 4$).

	\subsection{Nonemptieness of the curve $C_{en}$} \label{nonempty}
	For a generic polarized K3 surface $(X,L)$, it is known by \cite{MoriMukai} that $|L|$ always contains an integral nodal rational curve, hence $V_{L,g-1}$ is nonempty as well and the construction of the curve of elliptic nodes $C_{en}$ (see \Cref{def_Cd}) is valid. Although hyperelliptic K3 surfaces are not generic in the moduli space $\mcf_g$ and hence a priori the curve $C_{en}$ might be possibly empty,  the recent work  \cite{chengounelasliedtke2022curves} by Chen, Gounelas and Liedtke   concerning the existence of rational curves on lattice polarized K3 surfaces says that this is not the case (in the generic sense).

	Before going to the proof for the non-emptiness of the curve of elliptic nodes $C_{en}$ on hyperelliptic K3 surfaces, we fix some relevant conventions and review some results on lattice polarized K3 surfaces.
	
	\begin{defn}
		Let $\Lambda$ be a lattice, i.e.\ a free $\mbz$-module of finite rank equipped with a symmetric bilinear form. An \textit{ample $\Lambda$-polarized} K3 surface is a pair $(X, j)$, where $X$ is a K3 surface and $j \colon \Lambda \hookrightarrow \pic(X)$ is a primitive embedding such that $j(\Lambda)$ contains an ample line bundle.
	\end{defn}

	From now on, let $\Lambda_{g,k}$ be the lattice of rank two with intersection matrix
	\begin{equation} \label{special_lattice}
		\begin{bmatrix}
			2g-2 & 2k \\
			2k & 0
		\end{bmatrix},
	\end{equation}
	where $g > 1$ and $k > 0$. Then by \cite{Dolgachev1995MirrorSF} the moduli space $M_{\Lambda_{g,k}}$ of (ample)  $\Lambda_{g,k}$-polarized K3 surfaces is a quasi-projective variety of dimension 18. 
	
	\begin{rmk}
		Observe that the Picard lattice of a hyperelliptic K3 surface $(X,L_r)$ in $\mcf_{n,r}$ always contains  a lattice as above. Recall that by construction, $X$ admits a double cover $\pi \colon X \ra \mbf_{n}$ and $L_{r} = \pi^{\ast}\mso(e+rf)$, where $e,f$ are the standard generators of $\pic({\mbf_n})$ with relations in  (\ref{eq_basis_relation}). Let $M \coloneqq \pi^{\ast} \mso(f)$. Then the sublattice $\left\langle L_r, M \right\rangle \subset \pic(X)$ has intersection matrix  
		\begin{equation*} \Lambda_{g,1} = 
			\begin{bmatrix}
				4r-2n & 2 \\
				2 & 0
			\end{bmatrix},
		\end{equation*}
	where $g = 2r-n+1$ is the genus of the polarized K3 surface $(X,L_r)$.
	In particular, if $(X,L_r) \in \mcf_{n,r}$ is very general, then $\pic(X) \simeq \Lambda_{g,1}$.
	\end{rmk}
	The following result generalizes the existence result of integral nodal rational curves on a generic polarized K3 surface $(X,L) \in \mcf_g$ to a generic $\Lambda_{g,1}$-polarized K3 surface.
	\begin{thm}[{\cite[Thm.~3.1]{chengounelasliedtke2022curves}}] \label{lattice_polarized_rational_existence}
		Let $\Lambda$ be a lattice of rank two with intersection matrix as in \eqref{special_lattice} and suppose $L \in \Lambda$ is big and nef on a generic K3 surface $X \in M_{\Lambda}$. Then there exists an open dense subset $U \subset M_{\Lambda}$ such that for any $X \in U$, the linear system $|L|$ contains an integral nodal rational curve.
	\end{thm}
	It follows from the deformation theory of nodal curves on K3 surfaces (see \cite{MoriMukai} and also \cite[Ch.~13]{K3book_Huybrechts2016}) that one can smooth  any nodal point of such a nodal rational curve and obtain a one-dimensional family of nodal elliptic curves, hence we obtain directly the following result:
	
	\begin{cor} \label{nonemptiness}
		Let $n = 0$ or $1$, $r>n$ and $g = 2r+1-n$. Then for a generic hyperelliptic K3 surface $(X,L_r)$ in the $18$-dimensional locus  $\mcf_{n,r} \subset \mcf_{g}$ (see \Cref{moduli_dim_heK3}), the curve of elliptic nodes $C_{en}$ of $(X,L_{r})$ is nonempty.
	\end{cor}
	\subsection{Proof of  \Cref{intro_conj} for generic hyperelliptic K3} \label{proof}
	Use the notations in \Cref{dbcover_is_HEK3} and  \ref{moduli_dim_heK3}. Let $(X,L_r) \in \mcf_{n,r}$ be a generic hyperelliptic K3 surface admitting a double cover $\pi \colon X \ra \mbf_n$ with branch locus $B \in |-2K_{\mbf_n}|$, where $n = 0$ or $1$ and $r > n$. In addition, one can assume without loss of generality that $B$ is irreducible, since this is the generic case.
	
	We have just seen from \Cref{nonemptiness} that in our setting, the curve of elliptic nodes $C_{en}$ of $(X,L_{r})$ is nonempty. Now we show that $C_{en} \subset X$ is a constant cycle curve. 
	Similar to the double plane case in \Cref{geointer_dbplane}, one has a bijection between the following two sets 
	\begin{enumerate}[label=(\roman*)]
		\item lines in $|e+rf|$ that are simply tangent to $B$ at $(g-1)$ points and intersect $B$ transversally at the remaining intersection points
		\item  nodal elliptic curves in $|L_{r}|$
	\end{enumerate} 
	via the pullback $\pi^\ast \colon |e+rf| \ra |L_{r}|$ of linear systems. Moreover, the $(g-1)$ tangent points of such a line $l$ with $B$ are pulled back to the $(g-1)$ nodal points of $\pi^{\ast}l$, which  belong to the ramification locus $\pi^{-1}(B)$. Therefore, the curve  of elliptic nodes $C_{en}$ of  $(X,L_r)$ is exactly the curve $\pi^{-1}(B)$. Note that  $\pi^{-1}(B) \simeq B$ is a smooth, irreducible curve of genus $9$.
	
	Finally, since $\mbf_{n}$ is rational, for the same reason as in \Cref{exmp_intro}, the involution on $X$ induced by the double cover $\pi$ is non-symplectic with fixed locus $\pi^{-1}(B) = C_{en}$, hence $C_{en}$ is a constant cycle curve. We summarize the result below. 
	\begin{prop} \label{ccc_proof_F0}
		Let $(X,L_r)$ be a generic polarized K3 surface of genus $g = 2r-n+1$ in $\mcf_{n,r}$, where $n \in \{0,1\}$ and $r > n$. Let $\pi \colon X \ra \mbf_n$ be the double cover structure. Then the curve of elliptic nodes $C_{en}$ of $(X,L_r)$ is exactly the ramification locus of $\pi$. In particular, $C_{en}$ is a smooth constant cycle curve of genus $9$. 
	\end{prop}
	
	This together with the dimension computation in \Cref{moduli_dim_heK3} implies \Cref{mainthm}, namely \Cref{intro_conj} is true for an $18$-dimensional family in each moduli space $\mcf_g$ for $g \geqslant 3$.
	
	\section{Further Discussions on the Conjecture} \label{discussion}
	
	Throughout this section, let $(X,L)$ be a polarized K3 surface of genus $g > 2$ with nonempty Severi variety $V_{L,g-1}$. Recall from \Cref{def_Cd} that  the  curve of elliptic nodes $C_{en}$ of $(X,L)$ generically consists of singularities of nodal elliptic curves in $|L|$.  We have seen from \Cref{easyobserv} that the following condition is sufficient to prove that  $C_{en}$ is a constant cycle curve:
	\begin{center}
	($\ast$) For each nodal elliptic curve $E$ in $|L|$ with normalization $\nu \colon \tilde{E} \ra E$, one has\\ $[P_{i}^{\prime}]+[P_{i}^{\prime \prime}] = [P_{j}^{\prime}]+[P_{j}^{\prime \prime}]$ in $\ch_0(\tilde{E})$ for any $i \neq j$,\\ where $\{P_{i}^{\prime}, P_{i}^{\prime \prime}\}$ is the preimage of the $i$-th nodal point of $E$ under $\nu$. 
	\end{center}

	The main goal of this section is to explain that the sufficient condition ($\ast$)  is too strong to be necessary for $C_{en}$ being a constant cycle curve, in the sense that under this assumption $(X,L)$ almost has to be a hyperelliptic K3 surface (see \Cref{def_HEK3}).

	Let $(X,L)$ be as above. Recall that nodal elliptic curves in $|L|$ form a one-dimensional family over the Severi variety $V_{L,g-1}$. 
	We restrict this family to one of the irreducible components of $V_{L,g-1}$ when $V_{L,g-1}$ is not irreducible. After passing to a finite base change, we obtain a new family $p \colon \mathcal{E} \ra B$ with the following properties:
	\begin{itemize}
		\item The base $B$ is a smooth quasi-projective irreducible curve.
		\item Each fibre $\mce_t \coloneqq p^{-1}(t)$ over a closed point $t \in B$ is a nodal elliptic curve in $|L|$.
		\item The morphism $p$ has $(g-1)$ disjoint sections $\mcd_1, \ldots, \mcd_{g-1}$, whose  restrictions to each fibre $\mce_t$ correspond exactly to its $(g-1)$ nodal points.
	\end{itemize}
	The $(g-1)$ sections $\mcd_1, \ldots, \mcd_{g-1}$ endow the $(g-1)$ nodal points of $\mce_t$ with a natural labeling and we write  $P_{t,i}$ for the $i$-th node of $\mce_t$. 
	
	Since $p \colon \mathcal{E} \ra B$ is a family of projective curves of the same genus over a smooth base, it admits a simultaneous normalization $\nu \colon \tilde{\mathcal{E}} \ra \mathcal{E}$  in the sense of \cite{DiazHarris88idealsassdeformations}, i.e.\ :
	\begin{itemize}
		\item $\nu$ is the normalization of $\mathcal{E}$ (as a surface).
		\item The composition $\tilde{p} \coloneqq p \circ \nu \colon \tilde{\mathcal{E}} \ra \mathcal{E} \ra B$ is a smooth family of elliptic curves, and it is fibrewise the normalization $\nu_{t} \colon \tilde{\mce}_{t} \ra \mce_t$ of the corresponding fibre of $p$ over $t \in B$. 
	\end{itemize}
	
	Let $\tilde{\mcd}_i \coloneqq \nu^{-1}(\mcd_i) \subset \tilde{\mathcal{E}}$. 
	Since $B$ is smooth and $p$ is a smooth morphism, the total space $\tilde{\mathcal{E}}$ is a smooth surface and hence $\tilde{\mcd}_i \subset \tilde{\mathcal{E}}$ is a Cartier divisor for each $1 \leqslant i \leqslant g-1$. Denote by $\{	P_{t,i}^{\prime},P_{t,i}^{\prime \prime}\}$ the preimage of the $i$-th node $P_{t,i} \in \mce_t$ under the normalization  $\nu_{t}$. Then by construction, the restriction of $\tilde{\mcd}_i$ to a fibre $\tilde{\mce}_t$ of $\tilde{p}$ is exactly the set $\{	P_{t,i}^{\prime},P_{t,i}^{\prime \prime}\}$. 
	 
	 \begin{rmk}
	 	If $\tilde{\mcd}_1 \sim \cdots \sim \tilde{\mcd}_{g-1}$, that is, they are linearly equivalent to each other, then the condition ($\ast$) holds and by \Cref{easyobserv} the curve of elliptic nodes $C_{en}$ of $(X,L)$ is a constant cycle curve. Conversely, starting from the condition ($\ast$), one can only assert that $\tilde{\mcd_i} \sim \tilde{\mcd_j}$ up to a vertical divisor in the family $\tilde{p} \colon \tilde{\mce} \ra B$.
	 \end{rmk}

	 The following result indicates that at least for a generic polarized K3 surface $(X,L)$, one should not expect that the equations of zero cycles already hold at the level of Chow groups of elliptic curves.
	\begin{prop}
		Consider the construction above for a polarized K3 surface $(X,L)$ of genus $g>2$ with nonempty $V_{L,g-1}$. If the Cartier divisors $\tilde{\mcd}_1, \ldots, \tilde{\mcd}_{g-1}$ on $\tilde{\mathcal{E}}$ are linear equivalent to each other, then $X$ admits a double cover (induced by the linear system $|L|$).
		
		In other words, under our assumption, $X$ is a hyperelliptic K3 surface and can never be  generic in the moduli space $\mathcal{F}_g$.
	\end{prop}
	
	\begin{proof}
		According to
		\cite[Thm.~3.1]{saintdonat74projmodelofK3}, the existence of an irreducible curve in $|L|$ ensures that the linear system $|L|$ is base point free, hence $|L|$ induces a morphism $\phi_{L} \colon X \ra \mbp^g$. Fix a nodal elliptic curve $E \coloneqq \mce_t$ in $|L|$. Denote its $i$-th node by $P_i$ and let  $\{P_{i}^{\prime}, P_{i}^{\prime \prime}\}$ be the preimage of $P_i$ under the normalization $\nu \colon \tilde{E} \ra E$.
		
		The restriction $L|_{E}$ of the line bundle $L$ to $E$ also has no base points and the induced morphism $\phi_{L|_{E}}$ is the restriction of $\phi_{L}$ to $E$. Consider the line bundle $M_i \coloneqq \mso_{\tilde{E}}(P_i^{\prime}+P_{i}^{\prime \prime})$. Since $\tilde{E}$ is smooth elliptic, the linear system $|M_i|$ is a $g_2^1$ and induces a  morphism $\phi_{M_{i}}$ of degree two onto $\mbp^1$. By construction, the morphism $\phi_{M_{i}}$ factors through $E$. Let $\nu_{g-1} \colon \mbp^1 \ra \mbp^{g-1}$ be the Veronese embedding. Then we have the following diagram
		\begin{equation} \label{diagram}
			\begin{tikzcd}
				\tilde{E} \arrow[r, "\nu"] \arrow[d, "\phi_{M_{i}}"] & E \arrow[r, "i_{E}"] \arrow[d, "\phi_{L|_{E}}"] & X \arrow[d, "\phi_{L}"] \\
				\mbp^1 \arrow[r, "\nu_{g-1}"]                        & \mbp^{g-1} \arrow[r, "i"]                       & \mbp^{g}               
			\end{tikzcd}
		\end{equation}
		with  the right square commutative.
		
		We claim that the outer rectangle of \eqref{diagram} also commutes (up to an automorphism of $\mbp^{g}$) and this is enough to imply the statement. In fact, assuming the claim, then every nodal elliptic curve $E \in |L|$ maps generically $2 \ \colon 1$ onto its image under the morphism $\phi_L$. Since the one-dimensional family $p \colon \mce \ra B$ of nodal elliptic curves in $|L|$ dominates $X$,  we know immediately from \Cref{saintdonat} that the morphism  $\phi_L$ is indeed a double cover onto its image and by definition  $(X,L)$ is a hyperelliptic K3 surface.

		Now we prove the claim. It suffices to show that the following two pullbacks of the line bundle $\mso_{\mbp^g}(1)$ 
		\begin{align*}
			&(i \circ \nu_{g-1} \circ \phi_{M_{i}})^{\ast}\mso_{\mbp^g}(1) \simeq  \phi_{M_{i}}^{\ast}(\mso_{\mbp^1}(g-1)) 
			\simeq  \mso_{\tilde{E}}((g-1)(P_i^{\prime}+P_{i}^{\prime \prime}))
			\\ &(\phi_{L} \circ i_{E} \circ \nu)^{\ast}\mso_{\mbp^g}(1) \simeq (i_{E} \circ \nu)^{\ast} \mso_X(E) \simeq \mso_{\tilde{E}}(\sum_{i=1}^{g-1}(P_i^{\prime} + P_{i}^{\prime \prime})) 
		\end{align*}
		along two different routes in the diagram \eqref{diagram} are isomorphic. Here  for the last equality, we apply again the double point formula in \cite[Thm.~9.3]{fulton1998intersection} for the morphism $i_E \circ \nu \colon \tilde{E} \ra X$ of smooth varieties and use the identifications  $$\ch^1(\tilde{E}) \simeq \pic(\tilde
		{E}), \ \ \ch^1(X) \simeq \pic(X).$$ Since we assume that $\tilde{\mcd}_1, \ldots, \tilde{\mcd}_{g-1}$ are linear equivalent to each other, so are their restrictions  $\tilde{\mcd_i}|_{\tilde{E}} = P_i^{\prime} + P_i^{\prime \prime}$ to the fibre $\tilde{E}$, hence   the two pullbacks of $\mso_{\mbp^g}(1)$ above are isomorphic and this completes the proof. 
	\end{proof}
	
	\begin{rmk}
		Note that when $V_{L,g-1}$ is not irreducible, then the argument above  involves a choice of the family of curves $p \colon \mce \ra B$ (recall that the base $B$ dominates an irreducible component of $V_{L,g-1}$). However, this choice plays no role in our argument above: For dimension reasons, the chosen family of nodal elliptic curves always dominates the K3 surface $X$ and our proof still works without any change.
	\end{rmk}
	\medskip
	
	\printbibliography
	
\end{document}